\documentclass[11pt,a4paper]{amsart}
\usepackage{amsmath,amsthm}

\newtheorem{thm}{Theorem}[section]
\newtheorem{lem}[thm]{Lemma}
\theoremstyle{definition}
\newtheorem{defn}[thm]{Definition}
\numberwithin{equation}{section}
\newcommand{\K}{\mathbb K}
\newcommand{\N}{\mathbb N}
\newcommand{\eps}{\varepsilon}

\newcommand{\xs}{x^\ast}
\newcommand{\Xs}{X^{\ast}}

\DeclareMathOperator{\re}{Re}
\begin{document}

\title[New applications of extremely regular function spaces]%
{New applications of extremely regular function spaces}%
\author[Abrahamsen]{Trond A.~Abrahamsen}
\author[Nygaard]{Olav Nygaard}
\author[P\~{o}ldvere]{M\"{a}rt P\~{o}ldvere}

\address{Department of Mathematics, University of Agder, Servicebox 422, 4604 Kristiansand, Norway.}
\email{Trond.A.Abrahamsen@uia.no}
\urladdr{http://home.uia.no/trondaa/index.php3}
\email{Olav.Nygaard@uia.no}
\urladdr{http://home.uia.no/olavn/}

\address{Institute of Mathematics and Statistics, University of Tartu, J.~Liivi 2, 50409 Tartu, Estonia}
\email{mart.poldvere@ut.ee}

\thanks{Research by the third-named author supported, in part,
by institutional research funding IUT20-57 of the Estonian Ministry of
Education and Research.}
%
%\thanks{}
%
%
% ----------------------------------------------------------------
\begin{abstract}
Let $L$ be an infinite locally compact Hausdorff topological space.
We show that extremely regular subspaces of $C_0(L)$ have very strong
diameter $2$ properties and,
for every real number $\eps$ with $0<\eps<1$, contain an
$\eps$-isometric copy of $c_0$.
If $L$ does not contain isolated points they even have the Daugavet
property, and thus contain an asymptotically isometric copy of
$\ell_1$.
\end{abstract}
\maketitle
% ----------------------------------------------------------------

%%%%%%%%%%%%%%%%%%%%%%%%%%%%%%%%%%%%%%%%%%%%%%%%%%%%%%%%%%%%%%%%%%%%%%%%%%%%%%%%%%%%%%%%%%%%%
\section{Introduction}
%%%%%%%%%%%%%%%%%%%%%%%%%%%%%%%%%%%%%%%%%%%%%%%%%%%%%%%%%%%%%%%%%%%%%%%%%%%%%%%%%%%%%%%%%%%%%

Throughout, let $L$ be an infinite locally compact Hausdorff topological space,
and denote as usual by $C_0(L)$ the Banach space of continuous
$\K$-valued functions on $L$ that ``vanish at infinity'', where $\K$
is the field of either real or complex numbers.

\begin{defn}[Cengiz \cite{Cengiz--PAMS}]\label{def:er}
An \emph{extremely regular function space} is a linear subspace
$\mathcal{A}$ of $C_0 (L)$ such that for every $x_0\in L$,
every real number~$\eps$ with $0<\eps<1$, and every open neighbourhood
$V$ of $x_0$, there exists an $f\in\mathcal{A}$ such that
\[
\|f\|=1=f(x_0)>\eps>\sup_{x\in L\setminus V}|f(x)|.
\]
\end{defn}

The interest in extremely regular function spaces came from their
importance in Banach--Stone type theorems.
An example, also due to Cengiz \cite{Cengiz--PAMS}, is as follows:
\emph{If $L_1$ and $L_2$ are locally compact Hausdorff topological spaces
such that there exists a linear isomorphism $\varphi$ from an
extremely regular subspace of $C_0 (L_1)$
onto such a subspace of $C_0(L_2)$ with $\|\varphi\|\|\varphi^{-1}\| <
2$, then $L_1$ and $L_2$ are homeomorphic}
(here $C_0(L_1)$ and $C_0(L_2)$ are complex spaces).
Properties of extremely regular function spaces were studied in \cite{Cengiz--Pacific}.

In this paper, we demonstrate that extremely regular function spaces
play a role in a quite recent theory of Banach spaces, namely that
involving \emph{Daugavet spaces, diameter 2 spaces, and octahedral
  spaces}. Let us briefly explain some main lines of this theory
before returning to extremely regular function spaces and our results.

Let $X$ be a Banach space and $B_X$ its unit ball.
By a \emph{slice} of $B_X$ we mean a set of the form $S(\xs, \eps):=\{x\in B_X\colon \re\xs(x)>1-\eps\}$,
where $\xs$ is in the unit sphere $S_{\Xs}$ of $\Xs$ and $\eps > 0$.
A \emph{finite convex combination of slices} of~$B_X$ is a set $S$ of the form
$S=\sum_{i=1}^{n}\lambda_i S(\xs_i,\eps_i)$ where $n\in\N$, $\lambda_i>0$, $\sum_{i=1}^{n}\lambda_i=1$, $\xs_i \in S_{\Xs}$, and $\eps_i > 0$.

\begin{defn}\label{defn:diam2p2}
A Banach space $X$ has the \emph{strong diameter $2$ property}
(briefly, \emph{SD$2$P}) if every finite convex combination of slices
of $B_X$ has diameter~2.
\end{defn}

A lemma by Bourgain \cite[page~26, Lemma~II.1]{GGMS} says that
\emph{every non-empty relatively weakly open subset of $B_X$ contains
  a finite convex combination of slices}.
Thus the SD$2$P implies that every non-empty relatively weakly open
subset of $B_X$ has diameter $2$,
which in turn implies that every slice of $B_X$ has diameter $2$.
None of these implications is reversible (\cite{BG&L-P&RZ15--JMAA}, \cite{HL}).

It is an important observation of Deville and Godefroy from the late
1980s, stated without proof in \cite{G},
that $X$ having SD$2$P is equivalent to $\Xs$ being \emph{octahedral}.
A Banach space $Z$ is \emph{octahedral} if, for every
finite-dimensional subspace $F$ of $Z$ and every $\eps>0$,
there exists a $y\in S_Z$ such that
\[
\|x+ty\|\geq (1-\epsilon)(\|x\|+|t|)\quad\text{for every $x\in F$ and every $t\in\K$.}
\]
A complete proof of this equivalence can be found in \cite[Corollary
2.2]{BG&L-P&RZ14--JFA}
(the proof is carried out for the real case, but it is not too hard to
see that the result holds also in the complex case).

\begin{defn} \label{defn:diam2pmore}
A Banach space $X$
  \begin{enumerate}
\item
is \emph{almost square} (briefly, \emph{ASQ}) if whenever $n\in\N$ and $x_1,\dotsc,x_n\in S_X$,
there exists a sequence $(y_k)$ in $B_X$ such that $\|x_i \pm y_k\|\xrightarrow[k\to\infty]{}1$
for every $i\in\{1,\dotsc,n\}$ and $\|y_k\|\xrightarrow[k\to\infty]{}1$.
\item
has the \emph{symmetric strong diameter $2$ property} (briefly, \emph{SSD$2$P})
if whenever $n\in\N$, $S_1,\dotsc,S_n$ are slices of $B_X$, and $\eps>0$,
there exist $x_i\in S_i$, $i=1,\dotsc,n$, and $y\in B_X$ such that $x_i\pm y\in S_i$ for every $i\in\{1,\dotsc,n\}$ and $\|y\|>1-\eps$.
\end{enumerate}
\end{defn}

ASQ Banach spaces were studied in \cite{ALL}. The SSD$2$P has not been fully explored yet, but can be found in \cite[Lemma~4.1]{ALN1},
where it is observed that the SD$2$P is implied by the SSD$2$P.
In turn, it is not too hard to show that ASQ Banach spaces have the SSD$2$P.
On the other hand, the space $L_1[0,1]$ has the SD$2$P, but not the SSD$2$P,
and the space $C[0,1]$ has the SSD$2$P, but is not ASQ.

The property of a Banach space to be ASQ---and also the SSD$2$P---is rather strong.
The widest class of spaces known to be ASQ are non-reflexive $M$-embedded spaces~\cite[Corollary 4.3]{ALL}.
Also, $c_0(X)$ is ASQ for any Banach space $X$.
The widest class of spaces known to have the SSD$2$P are uniform algebras \cite[Theorem~4.2]{ALN1}.
Also, $\ell_\infty(X)$ has the SSD$2$P for any Banach space~$X$.

Let us also relate the \emph{Daugavet property} to the diameter $2$ properties.
Recall that a bounded linear operator $T$ on a Banach space $X$ is said to satisfy the \emph{Daugavet equation} if $\|I+T\|=1+\|T\|$.
Daugavet himself discovered in \cite{D} that every compact operator on $C[0,1]$ satisfies this equation,
thus initiating a very important topic in the theory of Banach spaces.
In \cite{L}, Lozanovski\u{\i} obtained the analogous result for $L_1[0,1]$.

\begin{defn}\label{defn:daugavet}
A Banach space $X$ has the \emph{Daugavet property} if every rank~$1$ operator on $X$ satisfies the Daugavet equation.
\end{defn}

\noindent%
Note that if a Banach space $X$ has the Daugavet property, then, in fact, every weakly compact operator on $X$ enjoys the Daugavet equation
(see, e.g., \cite[Theorem 2.3]{KShSW} or \cite[Theorem 2.7]{W}).

Towards the end of the 1990s, the Daugavet property was described in geometrical terms
(see \cite[Lemmas 2.1 and 2.2]{KShSW}, \cite[Lemmas 2 and 3]{Shv}, and \cite[Lemmas 2.2--2.4]{W}).
Spaces with the Daugavet property have the SD$2$P \cite[Theorem 4.4]{ALN1} and are octahedral \cite[Corollary 2.5]{BG&L-P&RZ14--JFA}.

Finally, we can announce our main results:
\emph{An extremely regular subspace of $C_0(L)$}
\begin{itemize}
\item \emph{has the SSD$2$P} (Theorem \ref{thm: A somewhat reg. => A has SSD2P});
\item \emph{is ASQ whenever $L$ is non-compact} (Theorem \ref{thm:asq});
\item \emph{has the Daugavet property whenever $L$ does not contain isolated points}
(Theorem \ref{thm: L does not have isolated points => s.w. reg. A is Daugavet});
\item \emph{contains an $\eps$-isometric copy of $c_0$ whenever $0<\eps<1$} (Theorem \ref{thm: somewhat regular subspace contains c_0}).
\end{itemize}

\noindent%
In fact, we prove these results for a wider class of subspaces of $C_0(L)$ than extremely regular ones,
that we call \emph{somewhat regular} subspaces (see Definition \ref{def: somewhat regular} below).

Throughout the paper, it should not make any confusion to denote, for a functional $\mu\in C_0(L)^\ast$,
its representing (regular) Borel measure also by $\mu$.

%%%%%%%%%%%%%%%%%%%%%%%%%%%%%%%%%%%%%%%%%%%%%%%%%%%%%%%%%%%%%%%%%%%%%%%%%%%%%%%%%%%%%%%%%%%%%
\section{Diameter $2$ properties for subspaces of~$C_0(L)$}
%%%%%%%%%%%%%%%%%%%%%%%%%%%%%%%%%%%%%%%%%%%%%%%%%%%%%%%%%%%%%%%%%%%%%%%%%%%%%%%%%%%%%%%%%%%%%

\begin{defn}\label{def: somewhat regular}
We call a linear subspace $\mathcal{A}$ of $C_0(L)$ \emph{somewhat regular,}
if, whenever $V$ is a non-empty open subset of $L$ and $0<\eps<1$, there is an $f\in\mathcal{A}$ such that
\begin{equation}\label{eq: ||f||=1, |f|=<eps off V}
\|f\|=1
\quad\text{and}\quad
\text{$|f(x)|\leq\eps$ for every $x\in L\setminus V$.}
\end{equation}
\end{defn}

\noindent
Notice that, in this case, $|f(x_0)|=1$ for some $x_0\in V$,
thus one may choose an $f\in\mathcal{A}$ satisfying \eqref{eq: ||f||=1, |f|=<eps off V} so that $f(x_0)=1$ for some $x_0\in V$.

It is clear that extremely regular subspaces of $C_0(L)$ are somewhat regular.
On the other hand, whenever $x_0$ is an accumulation point of $L$,
the subspace $\{f\in C_0(L)\colon\, f(x_0)=0\}$ of $C_0(L)$ is somewhat regular by courtesy of Urysohn's lemma,
but fails to be extremely regular.
Thus the class of somewhat regular subspaces of $C_0(L)$ is strictly
larger than that of extremely regular ones.

\begin{thm}\label{thm: A somewhat reg. => A has SSD2P}
Somewhat regular linear subspaces of $C_0(L)$ have the SSD$2$P.
\end{thm}

\noindent%
Theorem \ref{thm: A somewhat reg. => A has SSD2P} is a corollary from the following lemma.

\begin{lem}\label{lem: A somewhat reg; provided f_j and mu_i}
Let  $\mathcal{A}$ be a somewhat regular linear subspace of $C_0(L)$,
and let $n,m\in\N$, $f_1,\dotsc,f_n\in B_{\mathcal{A}}$,
$\mu_1,\dotsc,\mu_m\in B_{C_0(L)^{\ast}}$, and $\eps>0$. Then there
are $g_1,\dots,g_n,\phi\in B_{\mathcal{A}}$
such that, for every $j\in\{1,\dotsc,n\}$,
\begin{itemize}
\item[{\rm(1)}]
$|\mu_i(f_j-g_j)|<\eps$ for every $i\in\{1,\dotsc,m\}$;
\item[{\rm(2)}]
$|\mu_i(\phi)|<\eps$ for every $i\in\{1,\dotsc,m\}$;
\item[{\rm(3)}]
$\|\phi\|>1-\eps$;
\item[{\rm(4)}]
$\|g_j\pm\phi\|\leq1$.
\end{itemize}
\end{lem}

When dealing with subspaces of $C_0(L)$, the main challenge is often to find a substitute for Urysohn's lemma
(see, e.g., \cite[Section 2]{CGK}).
The following lemma---which the proofs of both Lemma \ref{lem: A somewhat reg; provided f_j and mu_i}
and Theorem \ref{thm: L does not have isolated points => s.w. reg. A
  is Daugavet} below rely on---is a ``Urysohn's lemma'' for somewhat
regular subspaces of $C_0(L)$. The lemma is inspired by \cite[proof of Theorem 1]{NW}.

\begin{lem}[cf. {\cite[proof of Theorem 1]{NW}}]\label{lem:
    modification to Dirk and Olav}
Let  $\mathcal{A}$ be a somewhat regular linear subspace of $C_0(L)$,
let $V$ be a non-empty open subset of $L$, and let $0<\eps<1$. Then
there are an $x_0\in V$ and an $f\in\mathcal{A}$ such that
\begin{itemize}
\item[{\rm(1)}]
$f(x_0)=1\leq\|f\|\leq1+\eps$;
\item[{\rm(2)}]
$|1-f(x)|\leq1+\eps$ for every $x\in V$;
\item[{\rm(3)}]
$|f(x)|\leq\eps$ for every $x\in L\setminus V$.
\end{itemize}
\end{lem}
\begin{proof}
Let $0<\delta<1$ and let $n\in\N$ satisfy $2/n<\delta$. Putting $V_0:=V$, by courtesy of the somewhat regularity of $\mathcal{A}$,
one can recursively find points $x_1,\dotsc,x_n\in V$,
functions $g_1,\dotsc,g_n\in\mathcal{A}$, and nonvoid open subsets $V_0\supset V_1\supset\dotsb\supset V_n$ such that,
for every $j\in\{1,\dotsc,n\}$,
\begin{equation*}
x_j\in V_{j-1},
\qquad
g_j(x_j)=\|g_j\|=1,
\qquad
%\quad\text{and}\quad
\text{$|g_j(x)|\leq\delta$ for every $x\in L\setminus V_{j-1}$,}
\end{equation*}
and $V_j=\bigl\{x\in V_{j-1}\colon\,|g_j(x)-1|<\delta\bigr\}$; thus, in fact, $x_j\in V_j$. Defining $x_0:=x_n$ and
\[
g:=\frac{g_1+\dotsb+g_n}{n},
\]
one has $\|g\|\leq1$, $|g(x)|\leq\delta$ for every $x\in L\setminus V$, and
\[
|1-g(x)|\leq\frac1{n}\sum_{j=1}^n|1-g_j(x)|\quad\text{for every $x\in L$.}
\]
Now let $x\in V$. Put $k:=\max\bigl\{j\in\{0,1,\dots,n\}\colon\,x\in V_j\bigr\}$. For $1\leq j\leq k$, one has $|1-g_j(x)|<\delta$;
$\delta\leq|1-g_{k+1}(x)|\leq2$; and, for $k+2\leq j\leq n$, one has $|g_j(x)|\leq\delta$ and hence $|1-g_j(x)|\leq1+\delta$. Thus
\[
|1-g(x)|\leq\frac{(n-1)(1+\delta)+2}{n}<1+\delta+\frac2n<1+2\delta.
\]
Since $x_0=x_n\in V_{n}$, one has $|g_j(x_0)-1|<\delta$ for every $j\in\{1,\dotsc,n\}$ and thus $|g(x_0)-1|<\delta$.
Defining $f:=\bigl(1/g(x_0)\bigr)g$, it remains to observe that, taking, from the very beginning, $\delta$ to be ``small enough'',
the conditions (1)--(3) obtain,
because, since $|g(x_0)|>1-\delta$,
\[
1=f(x_0)\leq\|f\|=\frac{\|g\|}{|g(x_0)|}<\frac{1}{1-\delta},
\]
for every $x\in V$,
\begin{align*}
|1-f(x)|
=\frac{|g(x_0)-g(x)|}{|g(x_0)|}<\frac{|g(x_0)-1|+|1-g(x)|}{1-\delta}<\frac{1+3\delta}{1-\delta},
\end{align*}
and, for every $x\in L\setminus V$,
\[
|f(x)|=\frac{|g(x)|}{|g(x_0)|}<\frac{\delta}{1-\delta}.
\]
\end{proof}

\begin{proof}[Proof of Lemma \ref{lem: A somewhat reg; provided f_j and mu_i}]
Let $0<\delta<1/2$.
Since $L$ is infinite, there is a point $y\in L$ such that $\max\limits_{1\leq i\leq m}|\mu_i|\bigl(\{y\}\bigr)<\delta$;
hence, by the regularity of $\mu_1,\dots,\mu_m$ and the continuity of $f_1\dotsc,f_n$,
there is a non-empty open subset $V$ of $L$ such that
\[
\max\limits_{1\leq i\leq m}|\mu_i|(V)<\delta
\quad\text{and}\quad
\max\limits_{1\leq j\leq n}\sup\limits_{x,z\in V}|f_j(x)-f_j(z)|<\delta.
\]
Since, by our assumption, $\mathcal{A}$ is somewhat regular, there are $x_0\in V$ and $f\in\mathcal{A}$ satisfying
the conditions {\rm(1)}--{\rm(3)} of Lemma \ref{lem: modification to Dirk and Olav} with $\eps$ replaced by $\delta$.
For every $j\in\{1,\dotsc,n\}$, defining $\alpha_j:=f_j(x_0)$ and $h_j:=f_j-\alpha_j\,f\in\mathcal{A}$,
one has $h_j(x_0)=0$ and $\|h_j\|\leq1+2\delta$, because
\[
|h_j(x)|\leq
\begin{cases}
|f_j(x)-\alpha_j|+|\alpha_j|\,|1-f(x)|\leq1+2\delta,&\quad\text{if $x\in V$;}\\
|f_j(x)|+|\alpha_j|\,|f(x)|\leq 1+\delta,&\quad\text{if $x\in L\setminus V$.}
\end{cases}
\]
For every $j\in\{1,\dotsc,n\}$, defining $g_j:=(1-2\delta)h_j$, one has $\|g_j\|\leq1-4\delta^2$ and, for every $i\in\{1,\dots,m\}$,
since
\begin{align*}
|\mu_i(f)|
&\leq\int_{L\setminus V}|f|\,d|\mu_i|+\int_V|f|\,d|\mu_i|\leq\delta\,|\mu_i|(L\setminus V)+2|\mu_i|(V)<3\delta,
\end{align*}
also
\begin{align*}
|\mu_i(f_j-g_j)|\leq2\delta|\mu_i(f_j)|+(1-2\delta)|\alpha_j||\mu_i(f)|<5\delta.
\end{align*}
Choose an open neighbourhood $U\subset V$ of $x_0$ such that
\[
\max\limits_{1\leq j\leq n}\sup\limits_{x\in U}|g_j(x)|\leq\delta.
\]
Since $\mathcal{A}$ is somewhat regular, there
%are $z_0\in U$ and
is a
$\psi\in\mathcal{A}$ such that
\[
%\psi(z_0)=
\|\psi\|=1\quad\text{and}\quad\text{$|\psi(x)|\leq4\delta^2$ for every $x\in L\setminus U$.}
\]
Put $\phi:=(1-\delta)\psi$. Then, for every $j\in\{1,\dotsc,m\}$, one has $\|g_j\pm\phi\|\leq1$, i.e., the condition (4) holds,
and, for every $i\in\{1,\dotsc,m\}$,
\[
|\mu_i(\phi)|
\leq\int_{L\setminus V}|\phi|\,d|\mu_i|+\int_V|\phi|\,d|\mu_i|\leq4\delta^2\,|\mu_i|(L\setminus V)+|\mu_i|(V)< 5\delta.
\]
Thus, one observes that taking, from the very beginning, $\delta$ to be ``small enough'', also the conditions (1)--(3) hold.
\end{proof}

If the space $L$ is non-compact, a stronger statement than that of Theorem~\ref{thm: A somewhat reg. => A has SSD2P} is true.

\begin{thm}\label{thm:asq}
Assume that $L$ is non-compact. Then every somewhat regular linear subspace of $C_0(L)$ is ASQ.
\end{thm}
\begin{proof}
Let $\mathcal{A}$ be a somewhat regular linear subspace of $C_0(L)$, and let $n\in\N$, $f_1,\dotsc,f_n\in S_{\mathcal{A}}$, and $\eps>0$.
It suffices to find an $f\in\mathcal{A}$ such that $\|f\|=1$ and
\begin{equation}\label{eq: ||f_j +- f|| =< 1+eps}
\|f_j\pm f\|\leq1+\eps
\quad\text{for every $j\in\{1,\dotsc,n\}$.}
\end{equation}
To this end, observe that the sets $K_j:=\bigl\{x\in L\colon\,|f_j(x)|\geq\eps\bigr\}$, $j=1,\dots,n$, are compact;
thus also their union $K:=\bigcup_{j=1}^n K_j$ is compact, and its complement $V:=L\setminus K$ is non-empty and open.
By the somewhat regularity of~$\mathcal{A}$, there is an $f\in\mathcal{A}$ satisfying \eqref{eq: ||f||=1, |f|=<eps off V}.
This $f$ also satisfies \eqref{eq: ||f_j +- f|| =< 1+eps}.
\end{proof}

Our next result produces examples of spaces with the Daugavet property.

\begin{thm}\label{thm: L does not have isolated points => s.w. reg. A is Daugavet}
Assume that $L$ does not contain isolated points. Then every somewhat regular linear subspace of $C_0(L)$ has the Daugavet property.
\end{thm}
\begin{proof}
Let $\mathcal{A}$ be a somewhat regular linear subspace of $C_0(L)$, let $g\in S_{\mathcal{A}}$,
let $\mu\in S_{C_0(L)^\ast}$ be such that $\|\mu|_{\mathcal{A}}\|=1$, and let $\alpha,\eps>0$.
In order for $\mathcal{A}$ to have the Daugavet property, by \cite[Lemma 2.2]{W} (or \cite[Lemma 2.2]{KShSW}),
it suffices to find a $\psi\in S_{\mathcal{A}}$ satisfying
\begin{equation}\label{eq: condition for Daugavetness for A }
\re\mu(\psi)>1-\alpha
\quad\text{and}\quad
\|g+\psi\|>2-\eps.
\end{equation}

To this end, let $\delta\in(0,1/3)$, let $\phi\in S_{\mathcal{A}}$ be such that $\re\mu(\phi)>1-\delta$,
let $y_0\in L$ be such that $|g(y_0)|=1$, and let an open neighbourhood $U$ of $y_0$ be such that
\[
|g(x)-g(y_0)|<\delta
\quad\text{and}\quad
|\phi(x)-\phi(y_0)|<\delta
\qquad\text{for all $x\in U$.}
\]
Since $y_0$ is not an isolated point, the set $U$ is infinite; thus there is a point $z_0\in U$ such that $|\mu|\bigl(\{z_0\}\bigr)<\delta$.
By the regularity of $\mu$, there is an open neighbourhood $V$ of $z_0$ such that $|\mu|(V)<\delta$.
One may assume that $V\subset U$ and thus
\[
|\phi(x)-\phi(z)|<2\delta
\quad\text{for all $x,z\in V$.}
\]
Since $\mathcal{A}$ is somewhat regular, there are $x_0\in V$ and $f\in\mathcal{A}$
satisfying the conditions {\rm(1)}--{\rm(3)} of Lemma \ref{lem: modification to Dirk and Olav} with $\eps$ replaced by $\delta$.
Put $h:=\phi-\phi(x_0)f$; then $h(x_0)=0$ and $\|h\|\leq1+3\delta$
(this can be shown as in the proof of Theorem \ref{thm: A somewhat reg. => A has SSD2P} for $h_j$); indeed,
\[
|h(x)|\leq
\begin{cases}
|\phi(x)-\phi(x_0)|+|\phi(x_0)|\,|1-f(x)|\leq1+3\delta,&\quad\text{if $x\in V$;}\\
|\phi(x)|+|\phi(x_0)|\,|f(x)|\leq 1+\delta,&\quad\text{if $x\in L\setminus V$.}
\end{cases}
\]
Since $h(x_0)=0$, there is an open neighbourhood $W$ of $x_0$ such that
\[
|h(x)|<\delta\quad\text{for all $x\in W$.}
\]
One may assume that $W\subset V$. Since $\mathcal{A}$ is somewhat regular, there are $w_0\in W$ and $\widehat{f}\in\mathcal{A}$
such that
\[
\widehat{f}(w_0)=\|\widehat{f}\|=1
\qquad\text{and}\qquad
|\widehat{f}(x)|\leq\delta\quad\text{for every $x\in L\setminus W$.}
\]
Putting $\widehat{\psi}:=h+g(w_0)\widehat{f}$, one has $\|\widehat{\psi}\|\leq1+4\delta$, because
\[
|\widehat{\psi}(x)|\leq|h(x)|+|\widehat{f}(x)|
\leq
\begin{cases}
\delta+1,&\quad\text{if $x\in W$;}\\
(1+3\delta)+\delta=1+4\delta,&\quad\text{if $x\in L\setminus W$.}
\end{cases}
\]
Since
\begin{align*}
\|\widehat{\psi}+g\|
&\geq|\widehat{\psi}(w_0)+g(w_0)|\geq2|g(w_0)|-|h(w_0)|\\
&\geq2|g(y_0)|-2|g(y_0)-g(w_0)|-|h(w_0)|\\
&>2-2\delta-\delta=2-3\delta,
\end{align*}
one has $\|\widehat{\psi}\|>1-3\delta$; thus, for $\psi:=\widehat{\psi}/\|\widehat{\psi}\|$, one has
\[
\|\widehat{\psi}-\psi\|=\biggl|1-\frac{1}{\|\widehat{\psi}\|}\biggr|\|\widehat{\psi}\|=\bigl|\|\widehat{\psi}\|-1\bigr|\leq4\delta,
\]
and hence
\[
\|g+\psi\|\geq\|g+\widehat{\psi}\|-\|\widehat{\psi}-\psi\|>2-3\delta-4\delta=2-7\delta.
\]
One has
\begin{align*}
\re\mu(\widehat{\psi})
&=\re\mu(h)+\re g(w_0)\mu(\widehat{f})
=\re\mu(\phi)-\re\phi(x_0)\mu(f)+\re g(w_0)\mu(\widehat{f})\\
&>1-\delta-|\mu(f)|-|\mu(\widehat{f})|.
\end{align*}
Since
\begin{align*}
|\mu(f)|
&\leq\left|\int_{L}f\,d\mu\right|\leq\int_{L}|f|\,d|\mu|=\int_{V}|f|\,d|\mu|+\int_{L\setminus
  V}|f|\,d|\mu|\\
&\leq(1+\delta)|\mu|(V)+\delta|\mu|(L\setminus
  V)<(1+\delta)\delta+\delta=(2+\delta)\delta<3\delta,
\end{align*}
and, similarly, $|\mu(\widehat{f})|<2\delta$, it follows that
$\re\mu(\widehat{\psi})>1-6\delta$, and thus
\[
  \re\mu(\psi)=\frac{\re\mu(\widehat{\psi})}{\|\widehat{\psi}\|}\geq\frac{\re\mu(\widehat{\psi})}{1+4\delta}>\frac{1-6\delta}{1+4\delta}.
\]
Hence one observes that, choosing, from the very beginning, $\delta$
to be ``small enough'', the function $\psi$ meets the conditions
\eqref{eq: condition for Daugavetness for A }.
\end{proof}

%%%%%%%%%%%%%%%%%%%%%%%%%%%%%%%%%%%%%%%%%%%%%%%%%%%%%%%%%%%%%%%%%%%%%%%%%%%%%%%%%%%%%%%%%%%%%
\section{Containment of $c_0$ and $\ell_1$}
%%%%%%%%%%%%%%%%%%%%%%%%%%%%%%%%%%%%%%%%%%%%%%%%%%%%%%%%%%%%%%%%%%%%%%%%%%%%%%%%%%%%%%%%%%%%%

Let $X$ and $Y$ be normed spaces, and let $0<\eps<1$. Recall that a
linear surjection $T\colon\,X\to Y$ is called an
\emph{$\eps$-isometry} if
\[
(1-\eps)\|x\|\leq \|Tx\|\leq (1+\eps)\|x\|\quad\text{for every $x\in X$.}
\]

It is well known that $C_0(L)$ contains isometric copies of $c_0$ (see
e.g \cite[Proposition~4.3.11]{MR2192298}), and
the same is true for many of its subspaces. For the somewhat regular
linear subspaces of $C_0(L)$ we have the following theorem.

\begin{thm}\label{thm: somewhat regular subspace contains c_0}
Let $\mathcal{A}$ be a somewhat regular closed linear subspace of $C_0(L)$.
Then, whenever $0<\eps<1$, there is an $\eps$-isometry from $c_0$ onto
a closed linear subspace of $\mathcal{A}$.
\end{thm}
\begin{proof}
Let $0<\eps<1$. Choose pairwise disjoint nonvoid open subsets $U_j$,
$j\in\N$, of $L$.
Since $\mathcal{A}$ is somewhat regular, for every $j\in\N$, there are
an $x_j\in U_j$ and an $f_j\in\mathcal{A}$ such that
\[
f_j(x_j)=1
\qquad
\text{and}
\qquad
|f_j(x)|\leq\dfrac{\eps}{2^j}\quad \text{for every $x\in L\setminus V_j$.}
\]
Denoting by $c_{00}$ the linear subspace of finitely supported sequences in $c_0$,
let $S_0\colon\,c_{00}\to\mathcal{A}$ be the linear operator
satisfying $S_0e_j=f_j$ for every $j\in\N$
where $e_j$ are the standard unit vectors in $c_0$.
Observing that, whenever $a=\sum_{j=1}^n\alpha_j e_j\in S_{c_{00}}$
and $x\in L$, one has
\begin{align*}
|(S_0a)(x)|
&=\biggl|\sum_{j=1}^n\alpha_j
  f_j(x)\biggr|\leq\sum_{j=1}^n|f_j(x)|\leq1+\sum_{j=1}^n\dfrac{\eps}{2^j}<1+\eps
\end{align*}
(because $|f_j(x)|\leq\eps/2^j$ whenever $x\not\in U_j$, and there is
at most one $j\in\N$ such that $x\in U_j$
(in which case $|f_j(x)|\leq1$)),
thus $S_0$ is bounded and $\|S_0\|\leq1+\eps$. Letting $S\colon\,
c_0\to\mathcal{A}$ be the bounded linear extension of $S_0$,
one has $\|S\|\leq1+\eps$ as well, and it remains to observe that,
whenever $a=(\alpha_j)_{j=1}^\infty\in c_0$,
picking $k\in\N$ such that $|\alpha_k|=\|a\|$, one has
\begin{align*}
\|Sa\|
&=\biggl\|\sum_{j=1}^\infty\alpha_j
  f_j\biggr\|\geq\biggl|\sum_{j=1}^\infty\alpha_j f_j(x_k)\biggr|
\geq|\alpha_k|\,|f_k(x_k)|-\sum_{\substack{j=1\\j\not=k}}^\infty|\alpha_j|\,|f_j(x_k)|\\
&\geq\|a\|-\|a\|\sum_{j=1}^\infty\frac{\eps}{2^j}=(1-\eps)\|a\|,
\end{align*}
because, for $j\ne k$, one has $x_k\not\in V_j$ and thus $|f_j(x_k)|\leq\eps$.
\end{proof}

It is natural to ask about containment of $\ell_1$ in somewhat regular linear
subspaces of $C_0(L)$. If $L$ does not contain isolated points, we have from
Theorem \ref{thm: L does not have isolated points => s.w. reg. A is Daugavet}
and \cite[Theorem~2.9]{KShSW} that all somewhat regular linear subspaces
of $C_0(L)$ contain $\ell_1$ (even asymptotically isometric copies of
$\ell_1$). But, if $L$ contains isolated points, the picture is not so
clear. In this case there might be somewhat regular subspaces of
$C_0(L)$ which contain $\ell_1$ and other such subspaces which do
not. For an example, take $C(\beta \mathbb N)$ and its subspaces $X =
\{f \in C(\beta \mathbb N)\colon f(x) =0\,\text{ for every } x \in \beta \mathbb N
\setminus \mathbb N\}$ and $Y = \{f \in C(\beta \mathbb N)\colon f(y) = 0\}$
where $y \in \beta \mathbb N \setminus \mathbb N$ is a fixed element.
It is straightforward to show that both these subspaces are somewhat
regular. Moreover, $X$ is isometrically isomorphic to $c_0$ and $Y$ is
isomorphic to $C(\beta \mathbb N)$.

%%%%%%%%%%%%%%%%%%%%%%%%%%%%%%%%%%%%%%%%%%%%%%%%%%%%%%%%%%%%%%%%%%%%%%%%%%%%%%%%%%%%%%%%%%%%%
\providecommand{\bysame}{\leavevmode\hbox to3em{\hrulefill}\thinspace}
\providecommand{\MR}{\relax\ifhmode\unskip\space\fi MR }
\providecommand{\MRhref}[2]{%
  \href{http://www.ams.org/mathscinet-getitem?mr=#1}{#2}
}
\providecommand{\href}[2]{#2}
%%%%%%%%%%%%%%%%%%%%%%%%%%%%%%%%%%%%%%%%%%%%%%%%%%%%%%%%%%%%%%%%%%%%%%%%%%%%%%%%%%%%%%%%%%%%%

%%%%%%%%%%%%%%%%%%%%%%%%%%%%%%%%%%%%%%%%%%%%%%%%%%%%%%%%%%%%%%%%%%%%%%%%%%%%%%%%%%%%%%%%%%%%%

%%%%%%%%%%%%%%%%%%%%%%%%%%%%%%%%%%%%%%%%%%%%%%%%%%%%%%%%%%%%%%%%%%%%%%%%%%%%%%%%%%%%%%%%%%%%%

\begin{thebibliography}{GGMS}

\bibitem{ALL}
T.~A. Abrahamsen, J.~Langemets, and V.~Lima,
\emph{Almost square Banach spaces},
J. Math. Anal. Appl. \textbf{434} (2016), no. 2, 1549--1565. \MR{3415738}

\bibitem{ALN1}
T.~A. Abrahamsen, V.~Lima, and O.~Nygaard,
\emph{Remarks on diameter 2 properties},
J. Convex. Anal. \textbf{20} (2013), no.~2, 439--452. \MR{3098474}


\bibitem{MR2192298}
F.~Albiac and N.~J. Kalton, \emph{Topics in {B}anach space theory}, Graduate
  Texts in Mathematics, vol. 233, Springer, New York, 2006. \MR{2192298
  (2006h:46005)}

\bibitem{BG&L-P&RZ14--JFA}
J.~Becerra~Guerrero, G.~L\'opez~P\'erez, and A.~Rueda Zoca,
\emph{Octahedral norms and convex combination of slices in Banach spaces},
J. Funct. Anal. \textbf{266} (2014), no.~4, 2424--2435. \MR{3150166}

\bibitem{BG&L-P&RZ15--JMAA}
J.~Becerra~Guerrero, G.~L\'opez~P\'erez, and A.~Rueda Zoca,
\emph{Big slices versus big relatively weakly open subsets in Banach spaces},
J. Math. Anal. Appl. \textbf{428} (2015), no. 2, 855--865. \MR{3334951}

\bibitem{CGK}
B. Cascales, A. J. Guirao, and V. Kadets,
\emph{A Bishop--Phelps--Bollob\'as type theorem for uniform algebras},
Adv. Math. \textbf{240} (2013), 370--382. \MR{3046314}

\bibitem{Cengiz--PAMS}
B. Cengiz, \emph{A generalization of the Banach--Stone theorem},
Proc. Amer. Math. Soc. \textbf{40} (1973), no.~2,
426--430. \MR{0320723 (47 $\sharp$9258)}

\bibitem{Cengiz--Pacific}
B. Cengiz,
\emph{On extremely regular function spaces},
Pacific J. Math., \textbf{49} (1973), no.~2, 335--338. \MR{0358317
  ((50 $\sharp$10783))}
%\MR{2784780 (2012d:46024)}

\bibitem{D}
I. K. Daugavet,
\emph{A property of completely continuous operators in the space C} (Russian),
Uspehi Mat. Nauk \textbf{18} (1963), no. 5 (113),
157--158. \MR{0157225 (28 $\sharp$461)}

\bibitem{GGMS}
N.~Ghoussoub, G.~Godefroy, B.~Maurey, and W.~Schachermayer,
\emph{Some topological and geometrical structures in Banach spaces},
Mem. Amer. Math. Soc. \textbf{70} (1987), no.~378, iv+116 pp.
\MR{0912637 (89h:46024)}

\bibitem{G}
G.~Godefroy,
\emph{Metric characterization of first Baire class linear forms and octahedral norms},
Studia Math. \textbf{95} (1989), no. 1, 1--15. \MR{1024271 (91h:46020)}

\bibitem{HL}
R.~Haller and J. Langemets,
\emph{Two remarks on diameter 2 properties},
Proc.  Est. Acad. Sci. \textbf{63} (2014), no.~1, 2--7.

\bibitem{KShSW}
V.~M. Kadets, R.~V. Shvidkoy, G.~G. Sirotkin, and D.~Werner,
\emph{Banach spaces with the Daugavet property},
Trans. Amer. Math. Soc. \textbf{352} (2000), no.~2, 855--873.
\MR{1621757 (2000c:46023)}

\bibitem{L}
G. Ya. Lozanovski\u{\i},
\emph{On almost integral operators in KB-spaces} (Russian. English summary),
Vestnik Leningrad. Univ. \textbf{21} (1966), no.~7,
35--44. \MR{0208375 (34 $\sharp$8185)}

\bibitem{NW}
O.~Nygaard and D.~Werner,
\emph{Slices in the unit ball of a uniform algebra},
Arch. Math. (Basel) \textbf{76} (2001), no.~6, 441--444.
\MR{1831500 (2002e:46057)}

\bibitem{Shv}
R.~V. Shvydkoy,
\emph{Geometric aspects of the Daugavet property},
J. Funct.  Anal. \textbf{176} (2000), no.~2, 198--212.
\MR{1784413 (2001h:46019)}

\bibitem{W}
D. Werner,
\emph{Recent progress on the Daugavet property},
Irish Math. Soc. Bull. (2001), no.~46, 77--97.
\MR{1856978 (2002i:46014)}

\end{thebibliography}
\end{document}